\documentclass{article}

\usepackage[latin1]{inputenc}
\usepackage{subfigure}
\usepackage{amsthm}
\usepackage{amsmath}
\usepackage{amsfonts}
\usepackage{wasysym}
\usepackage{enumerate}
\usepackage[margin=1in]{geometry}
\usepackage{bigints}
\usepackage[usenames]{color}
\usepackage{tikz,hyperref}
\newtheorem{lem}{Lemma}
\newtheorem{cor}[lem]{Corollary}
\newtheorem{thm}[lem]{Theorem}

\numberwithin{equation}{section}

\newcommand{\bfi}{\begin{figure} \begin{center}}
\newcommand{\efi}{\end{center} \end{figure}}
\newcommand{\capt}{\caption}
\newcommand{\hs}[1]{\hspace{#1}}

\newcommand{\da}{\hs{-2pt}\downarrow}
\newcommand{\lf}{\lfloor}
\newcommand{\rf}{\rfloor}

\providecommand{\scr}{\mathcal}

\author{Kenneth Barrese}
\title{Weighted File Placements on Singleton Boards}

\begin{document}
\maketitle

\begin{abstract}

In 2006, Briggs and Remmel gave a factorization theorem for $m$-level rook placements on singleton boards, a special subset of Ferrers boards. Subsequently, Barrese, Loehr, Remmel, and Sagan defined the $m$-weighted file placements to give a combinatorial interpretation to the aforementioned factorization theorem for all Ferrers boards. An unintended consequence of this definition is that the sum of the $m$-weights of file placements that are not $m$-level rook placements on a singleton board must be zero. In this paper, we attempt to illuminate this result by partitioning the set of file placements that are not $m$-level rook placements. We do so in such a way that it can be shown constructively that the sum of the $m$-weights on each partition must be zero using induction.

\end{abstract}

%


\section{Introduction}
\label{sec:in}

\subsection{History of Rook Theory}

In order to place this work in its historical context, we begin by tracing the relevant portion of rook theory through history. The field of rook theory, as a formal, academic subject, originated with a 1946 paper by Kaplansky and Riordan~\cite{kr:pra}. In 1975, Goldman, Joichi, and White were able to factor rook polynomials, the generating functions for rook numbers, of Ferrers boards by redefining them to be over the falling factorial basis of polynomials~\cite{gjw:rtI}.

Briggs and Remmel defined a generalization of rook placements, called $m$-level rook placements, in 2006~\cite{br:mrn}. They also provided a factorization of the $m$-level rook polynomial for a subset of Ferrers boards now called singleton boards. To factor $m$-level rook polynomials of all Ferrers boards, we were forced to abandon the elegant factorization which only depended on column heights, first introduced by Goldman, Joichi, and White. In a 2013 paper, Barrese, Loehr, Remmel, and Sagan factored $m$-level rook polynomials for all Ferrers boards~\cite{blrs:mrp}. The paper also introduced $m$-weighted file placements to give meaning to the elegant factorization that no longer yielded the $m$-level rook polynomial for non-singleton Ferrers boards.

Because the $m$-level rook polynomial and the $m$-weighted file placement polynomial agree on singleton boards, it turns out that the sum of the $m$-weighted file placements must be zero over all file placements that are not $m$-level rook placements. While this result is a consequence of the two factorizations yielding the same polynomial for singleton boards, no constructive explanation was produced for this phenomenon. The purpose of this paper is to explain why this is so, utilizing a new partition of the file placements that are not $m$-level rook placements on a singleton board. The new partition is constructed in such a way that the sum of the $m$-weighted file placements on each set in the partition will be zero, illuminating why the sum over the entire set of file placements that are not $m$-level rook placements must also be zero.

\subsection{Organization}

Section~\ref{sec:in} gives a brief overview of the history of $m$-level rook placements and lays out the goal of this paper before summarizing the structure of this paper. In Section~\ref{sec:rp} we lay out the formal definitions of classical rook placements, in the sense of Kaplansky and Riordan, and $m$-level rook placements, introduced by Briggs and Remmel. Along the way we present the original factorization theorem of Goldman, Joichi, and White, the factorization of $m$-level rook polynomials for singleton boards of Briggs and Remmel, and two formulations of the factorization theorem of $m$-level rook polynomials for general Ferrers boards by Barrese, Loehr, Remmel, and Sagan. The second of which was introduced in their 2016 paper~\cite{blrs:bom}.

Section~\ref{sec:wfp} introduces $m$-weighted file placements, including their motivation as a combinatorial object. The second half of the section formally describes why the sum of $m$-weighted file placements that are not $m$-level rook placements on a singleton board must be zero. In the penultimate section, Section~\ref{sec:par}, we develop the partition used to demonstrate why $m$-weighted file placements will sum to zero in the context in question. The section concludes with the theorem that the $m$-weighted file placements sum to zero on each individual set in that partition. Section~\ref{sec:open} presents three open projects that would be ideal considerations for mathematicians interested in taking the research presented in this paper further.

\section{Rook Placements}
\label{sec:rp}

\subsection{Ordinary Rook Placements}

Given a positive integer $n$, let $Sq_n$ denote an $n \times n$ array of square cells. A \emph{board}, $B$, is a finite subset of the cells of $Sq_n$ for some finite value of $n$. A \emph{rook placement of $k$ rooks} on $B$ is a subset of $B$ containing $k$ cells, such that no two cells are in the same row or column. Sometimes such a placement is called a \emph{non-attacking rook placement}, because if the rook chess piece were placed in each of the selected cells, no two would be able to attack each other under the standard rules of chess. Figure~\ref{exPlacement} gives an example of a rook placement of four rooks on $Sq_4$.

The \emph{$k$th rook number} of $B$, denoted $r_k(B)$, is the number of rook placements of $k$ rooks that exist on $B$. For example, if $B = Sq_4$ then $r_4(B) = 24 = 4!$. Consider that there are $4$ cells to place a rook in the first column, then three non-attacked squares in the second column, and so forth. The rook placement in Figure~\ref{exPlacement} is one of the $24$ rook placements of four rooks on $Sq_4$.

\bfi
\begin{tikzpicture}
\draw(-.5,1) node {$Sq_4=$};
\foreach \x in {0,.5,1,1.5,2}
   \draw (\x,0)--(\x,2);
\foreach \y in {0,.5,1,1.5,2}
	\draw (0,\y)--(2,\y);
\draw (.75, .75) node{$R$};
\draw (.25, 1.75) node{$R$};
\draw (1.25, .25) node{$R$};
\draw (1.75, 1.25) node{$R$};
\end{tikzpicture}
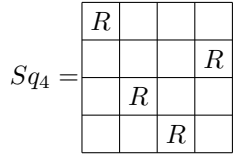
\capt{\label{exPlacement} A non-attacking rook placement of four rooks on $Sq_4$.}
\efi

For the rest of this paper we will consider a specific subset of boards with nice properties, called Ferrers boards. Given a non-negative integer partition $0 \leq b_1 \leq b_2 \leq \ldots \leq b_n$ the \emph{Ferrers board} $B=(b_1,b_2,\ldots,b_n)$ consists of the bottom $b_i$ cells in the $i$th column. Figure~\ref{exFerrers} gives an example of a board which is a Ferrers board and, for contrast, another board which is not. The main advantage of Ferrers boards is that their rook polynomials factor nicely.

\bfi
\begin{tikzpicture}
\draw(-.5,1) node {$B_1=$};
\foreach \x in {0,.5} 
   \draw (\x,0)--(\x,.5);
\foreach \x in {1} 
   \draw (\x,0)--(\x,1);
\foreach \x in {1.5,2}
   \draw (\x,0)--(\x,2);
\foreach \y in {0,.5}
	\draw (0,\y)--(2,\y);
\foreach \y in {1}
        \draw (1,\y)--(2,\y);
\foreach \y in {1.5,2}
          \draw(1.5,\y)--(2,\y);
\end{tikzpicture}
\hs{40pt}
\begin{tikzpicture}
\draw(-.5,1) node {$B_2=$};
\foreach \x in {0,.5} 
   \draw (\x,0)--(\x,1);
\foreach \x in {1} 
   \draw (\x,0)--(\x,1);
\foreach \x in {1.5,2}
   \draw (\x,0)--(\x,2);
\foreach \y in {0,.5}
	\draw (0,\y)--(2,\y);
\foreach \y in {1}
        \draw (1,\y)--(2,\y);
\foreach \y in {1.5,2}
          \draw(1.5,\y)--(2,\y);
\draw (0,1)--(.5,1);
\end{tikzpicture}
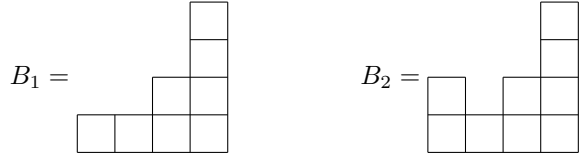
\capt{\label{exFerrers} The board on the left is the Ferrers board $B_1= (1,1,2,4)$. The board on the right is not a Ferrers board.}
\efi

Given a Ferrers board $B$, the \emph{rook polynomial}, as defined by Goldman, Joichi, and White~\cite{gjw:rtI}, of $B$ is $$p(B,x) = \sum_{k=0}^n r_k(B)x \da_{n-k},$$ where $n\da_k$ is called the \emph{$k$th falling factorial of n} and is defined by $n\da_k = \prod_{i=0}^{k-1} n-i$. Defining the rook polynomial using the falling factorials allowed Goldman, Joichi, and White to factor the rook polynomials of Ferrers boards, with roots related to the board's column heights, as follows.

\begin{thm}[Goldman-Joichi-White~\cite{gjw:rtI}]
\label{gjw}
If $B=(b_1,\dots,b_n)$ is a Ferrers board, then
$$
p(B,x) = \prod_{i=1}^n (x+b_i-(i-1)).
$$
\end{thm}

Continuing the example illustrated on the left side of Figure~\ref{exFerrers}, $B_1 = (1,1,2,4)$. This means that $p(B_1,x)=(x+1-0)\cdot(x+1-1)\cdot(x+2-2)\cdot(x+4-3)=x^2(x+1)^2=x^4+2x^3+x^2$. Two boards are called \emph{rook equivalent} if they have the same rook numbers for all values of $k\geq 0$. If the two Ferrers boards have the same number of columns, the boards are rook equivalent if and only if they have the same rook polynomials. Therefore Goldman, Joichi, and White's factorization theorem provides a quick test to determine whether two Ferrers boards are rook equivalent or not based only on their column heights.

\subsection{$m$-Level Rook Placements}

Briggs and Remmel defined a generalization of rook placements called $m$-level rook placements~\cite{br:mrn}. Given positive integers $m$ and $n$, let $Sq_{n,m}$ denote an $mn \times n$ array of square cells. We will partition the rows of $Sq_{n,m}$ into sets of size $m$ called \emph{levels}, where the first level contains the bottom $m$ rows, the second level contains rows $m+1$ through $2m$, all the way up to the $n$th level. For the purposes of $m$-level rook placements, we will consider a board to be a finite subset of the cells of $Sq_{n,m}$.

An \emph{$m$-level rook placement of $k$ rooks} on $B$ is a subset of $B$ containing $k$ cells, no two of which are in the same level or column. Notice that an $m$-level rook placement replaces the role of a row with that of a level, a set of rows. Clearly then every $m$-level rook placement is also a rook placement, and a rook placement (an ordinary rook placement) is equivalent to a $1$-level rook placement. The \emph{$k$th $m$-level rook number} of $B$, denoted $r_{k,m}(B)$, is the number of $m$-level rook placements of $k$ rooks on $B$.

\bfi
\begin{tikzpicture}
\draw(-.5,1) node {$B=$};
\foreach \x in {0} 
   \draw (\x,0)--(\x,.5);
\foreach \x in {.5,1} 
   \draw (\x,0)--(\x,1);
\foreach \x in {1.5,2}
   \draw (\x,0)--(\x,1.5);
\foreach \y in {0,.5}
	\draw (0,\y)--(2,\y);
\foreach \y in {1}
        \draw (.5,\y)--(2,\y);
\foreach \y in {1.5}
          \draw(1.5,\y)--(2,\y);
\end{tikzpicture}
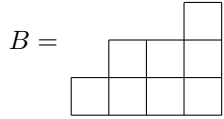
\capt{\label{nonsing} For $m=3$ the board $B= (1,2,2,3)$ is not a singleton board, because the first three columns all intersect the first level without filling it. However, if $m=2$, then $B$ is a singleton board since only the first column and fourth column intersect levels without filling them, and they each leave a different level incomplete.}
\efi

Many theorems regarding $m$-level rook placements are easiest to prove for a nice subset of Ferrers boards, called singleton boards. To define a singleton board we use the concept of the \emph{$m$-floor of $n$}, denoted $\lfloor n \rfloor_m$, the largest multiple of $m$ less than or equal to $n$. A Ferrers board $B = (b_1, b_2,\ldots,b_n)$ is a \emph{singleton board} if $b_i -\lfloor b_i \rfloor_m \neq 0$ implies that $\lfloor b_i \rfloor_m < \lfloor b_{i+1} \rfloor_m$. The descriptor, ``singleton,'' arises because this is equivalent to requiring that for each level, there is at most a single column of $B$ that intersects that level in more than one cell but less than $m$ cells. Note that the set of singleton boards does depend on the value of $m$ being considered, and further that when $m=1$, the set of singleton boards is equal to the set of Ferrers boards. For $m\geq2$, singleton boards are a proper subset of Ferrers boards. See Figure~\ref{nonsing} for an example of a Ferrers board that is singleton for $m=2$ but is not singleton for $m=3$.

One theorem which is more straightforward for singleton boards involves factoring the $m$-level rook polynomial of a singleton board $B$. In order to define the $m$-level rook polynomial, first we need the $m$-falling factorial. Similarly to the falling factorial, the $m$-falling factorial is defined by: $n\da_{k,m}=\prod_{i=0}^{k-1}n-mi$. Then, given a Ferrers board $B$, the \emph{$m$-level rook polynomial} of $B$ is $$p_m(B,x)= \sum_{k=0}^n r_{k,m}(B)x\da_{n-k,m.}$$ As before, if two Ferrers boards with the same number of columns have the same $m$-level rook polynomial, we say that the boards are $m$-level rook equivalent. Briggs and Remmel gave the following factorization theorem for the $m$-level rook polynomial of a singleton board.

\begin{thm}[Briggs-Remmel~\cite{br:mrn}]
\label{br factor}
If $B=(b_1,\ldots,b_n)$ is a singleton board, then
$$
p_m(B,x) = \prod_{i=1}^n(x+b_i-m(i-1)).
$$
\end{thm}

Notice that Theorem~\ref{gjw} is a special case of Theorem~\ref{br factor} when $m=1$.

This factorization was extended to all Ferrers boards by Barrese, Loehr, Remmel, and Sagan in a couple of ways. The first requires the definition of a \emph{zone}, $z = [s,t]$, which is a maximal range of column indices of $B$ such that $\lfloor b_s \rfloor_m = \lfloor b_{s+1} \rfloor_m = \ldots = \lfloor b_t \rfloor_m$. The \emph{remainder} of column $i$ is $\rho_i = b_i - \lfloor b_i \rfloor_m$, and the remainder of zone $z$ is $\rho_z = \sum_{k=s}^t \rho_k$. These definitions set up the following theorem.

\begin{thm}[\cite{blrs:mrp}]
\label{zone factor}
If $B=(b_1,\ldots,b_n)$ is a Ferrers board, then
$$
p_m(B,x) = \prod_{i=1}^n 
\begin{cases}
x+\lf b_j \rf_m -(i-1)m + \rho_z &\text{if $i$ is the last index in its zone $z$,}\\
x+\lf b_j \rf_m -(i-1)m &\text{otherwise.}
\end{cases}
$$
\end{thm}

Another generalization to all Ferrers boards comes from letting $l_j$ be the \emph{$j$-th level number} of $B$, or the number of cells of $B$ in the $j$th level from the top, that is, level $n+1-j$ since $B$ consists of $n$ levels in total. Using this definition yields another factorization theorem.

\begin{thm}[\cite{blrs:bom}]
\label{level factor}
If $B=(b_1,\ldots,b_n)$ is a Ferrers board, then
$$
p_m(B,x) = \prod_{j=1}^n (x+l_j-m(j-1)).
$$
\end{thm}

While Theorem~\ref{level factor} is more concisely written than Theorem~\ref{zone factor}, in both cases you will get the same set of factors, because $p_m(B,x)$ remains the same. And, in either case, some information is lost for non-singleton boards, as they cannot be uniquely determined by the roots of their factorization. This suggests the open question, given an ordered set of level numbers: $(l_1,l_2,\ldots,l_n)$, how many unique Ferrers boards have those specific level numbers.

\section{Weighted File Placements}
\label{sec:wfp}

\subsection{Motivation}

As noted in Theorem~\ref{br factor}, $\prod_{i=1}^n(x+b_i-m(i-1))$ only gives a factorization of the $m$-level rook polynomial if the board in question is a singleton board. As such, Barrese, Loehr, Remmel, and Sagan defined weighted file placements to give combinatorial meaning to the product from Theorem~\ref{br factor} for a general Ferrers board. A \emph{file placement} on board $B$ is a subset, $F$, of the cells of $B$ such that no two are in the same column. Note that, unlike a rook placement, any number of cells are allowed to be selected from the same row, as long as there remains at most one per column. Figure~\ref{exFile} shows a file placement on $B = (2,2,4,4,4,4)$.

\bfi
\begin{tikzpicture}
\draw(-.5,1) node {$B=$};
\foreach \x in {0,.5}
   \draw (\x,0)--(\x,1);
\foreach \y in {0,.5,1}
	\draw (0,\y)--(3,\y);
\foreach \x in {1,1.5,2,2.5,3}
   \draw (\x,0)--(\x,2);
\foreach \y in {1.5,2}
	\draw (1,\y)--(3,\y);
\draw (.25, .75) node{$R$};
\draw (1.25, 1.75) node{$R$};
\draw (2.25, 1.75) node{$R$};
\draw (1.75, .75) node{$R$};
\draw (2.75, 1.75) node{$R$};
\end{tikzpicture}
\capt{\label{exFile} A file placement on $B=(2,2,4,4,4,4)$.}
\efi

Let $f_j$ be the number of cells of $F$ in row $j$ of $B$ and assume that there are $n$ rows in $B$. For a fixed $m\geq 1$ we can define the \emph{$m$-weight} of $F$ as follows, $wt_m(F) = \prod_{j=1}^n 1\da_{f_j,m}$. For example, the placement in Figure~\ref{exFile} has $f_1 = 0$, $f_2 = 2$, $f_3 = 0$, and $f_4 = 3$. If we pick $m=3$ then $wt_3(F) = 1\da_{0,3}\cdot1\da_{2,3}\cdot1\da_{0,3}\cdot1\da_{4,3} =[(1)]\cdot [(1)(-2)]\cdot[(1)]\cdot[(1)(-2)(-5)] = -20$.

If $\mathcal{F}_k$ denotes the set of all file placements of $k$ rooks on a board, B, we define the \emph{$k$th $m$-weighted file placement number} of the board to be $$f_{k,m}(B) = \sum_{F\in\mathcal{F}_k} wt_m(F).$$ Using this definition yields the following theorem:

\begin{thm}[\cite{blrs:mrp}]
\label{file factor}
If $B=(b_1,\ldots,b_n)$ is a Ferrers board, then
$$
\sum_{k=0}^nf_{k,m}(B)x\da_{n-k,m} = \prod_{i=1}^n(x+b_i-(i-1)m).
$$
\end{thm}

\subsection{File Placements on Singleton Boards}

Continuing the previous convention, we use $F$ to denote a particular file placement and $\mathcal{F}$ to denote the set of all file placements on $B$. Notice that, while the set of file placements, $\mathcal{F}$, of a given board does not depend on the choice of $m$, the $m$-weighted file placement numbers certainly do. Also note that the right side in Theorem~\ref{file factor} is exactly the same as in Theorem~\ref{br factor} but the requirement that $B$ be a singleton board has been relaxed to only require that it is a Ferrers board.

Assume that file placement $F$ happens to also be an $m$-level rook placement on $B$. It follows that $wt_m(F) = 1$ since each $f_j$ will equal either $0$ or $1$ and $1\da_{0,m} = 1\da_{1,m} = 1$. From this, the following corollary is clear. I should note that it was known to the authors of~\cite{blrs:mrp}, although they do not seem to have included it in the final paper:

\begin{cor}[\cite{blrs:mrp}]
\label{singFile}
If $B=(b_1,\ldots,b_n)$ is a singleton board, then
$$
\sum_{k=0}^n f_{k,m}(B)x\da_{n-k,m} = \sum_{k=0}^n r_{k,m}(B)x\da_{n-k,m}.
$$

Further, if $F_k'$ is a file placement of $k$ rooks on $B$ that is not an $m$-level rook placement and $\mathcal{F'}_k$ is the set of all file placements of $k$ rooks that are not $m$-level rook placements, then $\sum_{\mathcal{F'}_k}wt_m(F'_k) = 0$.
\end{cor}

While this result is clearly a consequence of $\sum_{k=0}^n f_{k,m}(B)x\da_{n-k,m}$ and  $\sum_{k=0}^n r_{k,m}(B)x\da_{n-k,m}$ having the exact same factorization for a singleton board, it is not immediately clear why the $m$-weights of all the file placements that are not $m$-level rook placements should cancel out. The rest of this paper develops a method of partitioning these file placements so that the $m$-weights cancel out on each partition. This provides a constructive explanation for why they cancel out over the entire set.

\section{Partitions of File Placements}
\label{sec:par}

For a fixed $m\geq 2$, let $B$ be a singleton board with regard to the choice of $m$. Let $\mathcal{F}'$ denote the set of all file placements on $B$ that are not $m$-level rook placements. Partition $\mathcal{F}'$ as follows:

For $F_0 \in \mathcal{F}'$ there is at least one level containing multiple rooks, since otherwise $F_0$ would be an $m$-level rook placement. Out of the levels containing multiple rooks, identify the one containing the fewest rooks. If two or more such levels exist, choose the lowest such level, simply to have a single, well-defined, level. We will call this level $l$. The partition containing $F_0$ contains another given file placement if and only if both placements contain the same number of rooks, all the rooks outside of level $l$ are in the same cells in both placements, the leftmost rook in level $l$ is in the same cell in both placements, and the columns of the other rooks in level $l$ are the same in both placements. This is the same as allowing every rook but the leftmost in level $l$ to be placed in any cell in level $l$ as long as it remains in its original column. Call the partition containing $F_0$, $\mathcal{P}(F_0)$.

The choice to pick a level containing the fewest number of rooks greater than $1$ is not essential, it simply serves to reduce the size of the partition containing $F_0$. 

\begin{lem}
If $F_0$ is a file placement on $B$ that is not an $m$-level rook placement in which level $l$ contains $n$ rooks, then $\mathcal{P}(F_0)$ contains $m^{n-1}$ file placements.
\end{lem}

\begin{proof}
Level $l$ contains $n$ rooks, the leftmost of which must be fixed. The remaining $n-1$ rooks in level $l$ can be placed in any of the $m$ cells in its column. Since these choices are independent, there are $m^{n-1}$ ways to make them. All other rooks in the placement are fixed.
\end{proof}

Figure~\ref{parExample} shows the four file placements in the case where $m = 2$, $l=1$, and $n+1 = 3$.

\bfi
\begin{tikzpicture}
\foreach \x in {0} 
   \draw (\x,0)--(\x,.5);
\draw (.5,0)--(.5,1.5);
\foreach \x in {1,1.5,2,2.5,3,3.5}
   \draw (\x,0)--(\x,2);
\foreach \y in {0,.5}
	\draw (0,\y)--(3.5,\y);
\foreach \y in {1,1.5}
        \draw (.5,\y)--(3.5,\y);
\draw(1,2)--(3.5,2);
\draw[thick,dashed] (0,1)--(3.5,1);
\draw (.75, 1.25) node{$R$};
\draw (1.25, .75) node{$R$};
\draw (1.75, .25) node{$R$};
\draw (2.25, 1.75) node{$R$};
\draw (2.75, .25) node{$R$};
\draw (3.25, 1.25) node{$R$};
\end{tikzpicture}
\hs{10pt}
\begin{tikzpicture}
\foreach \x in {0} 
   \draw (\x,0)--(\x,.5);
\draw (.5,0)--(.5,1.5);
\foreach \x in {1,1.5,2,2.5,3,3.5}
   \draw (\x,0)--(\x,2);
\foreach \y in {0,.5}
	\draw (0,\y)--(3.5,\y);
\foreach \y in {1,1.5}
        \draw (.5,\y)--(3.5,\y);
\draw(1,2)--(3.5,2);
\draw[thick,dashed] (0,1)--(3.5,1);
\draw (.75, 1.25) node{$R$};
\draw (1.25, .75) node{$R$};
\draw (1.75, .75) node{$R$};
\draw (2.25, 1.75) node{$R$};
\draw (2.75, .25) node{$R$};
\draw (3.25, 1.25) node{$R$};
\end{tikzpicture}

\vspace{10pt}
\begin{tikzpicture}
\foreach \x in {0} 
   \draw (\x,0)--(\x,.5);
\draw (.5,0)--(.5,1.5);
\foreach \x in {1,1.5,2,2.5,3,3.5}
   \draw (\x,0)--(\x,2);
\foreach \y in {0,.5}
	\draw (0,\y)--(3.5,\y);
\foreach \y in {1,1.5}
        \draw (.5,\y)--(3.5,\y);
\draw(1,2)--(3.5,2);
\draw[thick,dashed] (0,1)--(3.5,1);
\draw (.75, 1.25) node{$R$};
\draw (1.25, .75) node{$R$};
\draw (1.75, .25) node{$R$};
\draw (2.25, 1.75) node{$R$};
\draw (2.75, .75) node{$R$};
\draw (3.25, 1.25) node{$R$};
\end{tikzpicture}
\hs{10pt}
\begin{tikzpicture}
\foreach \x in {0} 
   \draw (\x,0)--(\x,.5);
\draw (.5,0)--(.5,1.5);
\foreach \x in {1,1.5,2,2.5,3,3.5}
   \draw (\x,0)--(\x,2);
\foreach \y in {0,.5}
	\draw (0,\y)--(3.5,\y);
\foreach \y in {1,1.5}
        \draw (.5,\y)--(3.5,\y);
\draw(1,2)--(3.5,2);
\draw[thick,dashed] (0,1)--(3.5,1);
\draw (.75, 1.25) node{$R$};
\draw (1.25, .75) node{$R$};
\draw (1.75, .75) node{$R$};
\draw (2.25, 1.75) node{$R$};
\draw (2.75, .75) node{$R$};
\draw (3.25, 1.25) node{$R$};
\end{tikzpicture}
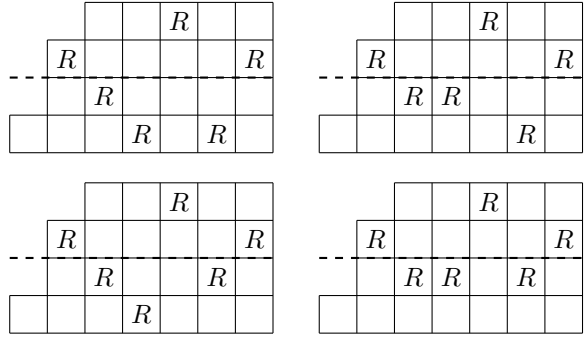
\capt{\label{parExample} The four file placements together in their specific partition for $m=2$ and $B=(1,3,4,4,4,4,4)$. The dashed line separates the two different levels.}
\efi

The following lemma considers what occurs when $l$ contains exactly two rooks. In addition to providing the base case for the inductive argument of our main theorem, it is illuminating in its own right to consider this case.

\begin{lem}
\label{base case}
Let $F_0$ be a file placement on singleton board, $B$, in which the level, $l$, defined as above contains exactly two rooks. Let $\mathcal{P}(F_0)$ denote the set obtained by allowing the rightmost rook in $l$ to take any position in its original column that is still in level $l$. Then: $$\sum_{F\in\mathcal{P}(F_0)} wt_m(F)=0.$$
\end{lem}

\begin{proof}
Since $B$ is a singleton board, we know that there is at most one column of $B$ that intersects level $l$ non-trivially but in fewer than $m$ cells. Furthermore, if such a column exists, it must be the leftmost column to intersect level $l$ non-trivially, because the column heights are weakly increasing. Since there is a rook to the left of the rook we are moving, the rook we are moving must be in a column that intersects level $l$ in a full $m$ cells.

Since there are exactly two rooks in level $l$, there will be one case in which the rook we are moving occupies the same row as another rook, and $m-1$ cases in which the rook we are moving is alone in its row. Any other rows in level $l$ must be empty, therefore the weight of those specific rows will be $1\da_{0,m}=1$, if they exist. In the $m-1$ cases where the two rooks are in different rows, each row containing a rook contributes $1\da_{1,m}=1$ to the product determining the overall weight of the board. In the unique case where the two rooks are in the same row, that row contributes $1\da_{2,m}=1\cdot(1-m) = 1-m$ to the overall weight.

For $F\in\mathcal{P}(F_0)$ let $\hat{wt}_m(l^C)$ be the product of the row weights for the rows not in level $l$. Notice that since we are only moving rooks in level $l$, for any board in $\mathcal{P}(F_0)$, $\hat{wt}_m(l^C)$ will be some constant, $W$, for any placement in $\mathcal{F_0}$. Therefore, when the two rooks in $l$ are in different rows, $wt_m(F) = W\cdot 1\cdot 1=W$ and when the two rooks in $l$ are in the same row $wt_m(F) = W\cdot (1-m)$, since level $l$ contributes $(1-m)$ to the overall product.

Considering that $\mathcal{P}(F_0)$ contains $m-1$ file placements where the two rooks are in different rows and a unique file placement where the two rooks are in the same row, we obtain:

$$\sum_{F\in\mathcal{P}(F_0)}wt_m(F) = (m-1)\cdot W + 1\cdot W\cdot(1-m) = W\left (m-1+1-m\right ) = 0.$$
\end{proof}

It would be nice if this argument generalized directly to file placements where the level containing multiple rooks with the fewest rooks has more than two rooks. Specifically, if we could pick one of the rooks in that row and allow that rook to move to any cell in its original level and column, then sum the $m$-weights of all the file placements we obtained and get $0$, that would be nice and simple. We could conclude that we can always group weighted file placements that are not $m$-level rook placements into partitions of size $m$ where the weight summed over the entire partition equals to zero. Unfortunately, as the following example illustrates, this is not the case.

Consider the file placement in the top left of Figure~\ref{parExample} and look at the rooks in the first level. If you pick either of the two rooks to the right and allow that specific rook to be in any cell in its original column and level, no matter where you put the rook it will be in a row with one other rook. Since $m=2$, each of those choices will contribute $1\da_{2,2}=-1$ to the overall $m$-weight. Thus, summing over all possible locations for the given rook, we get $-2W$ which certainly will not be $0$ when $m\neq 1$, since $m=1$ is the only case where $1\da_{k,m}=0$ and therefore $W=0$ is possible.

On the other hand, consider the leftmost rook in the bottom level of the file placement. Left in its current position that level contributes $1\da_{2,2}=-1$ to the overall product. However, if it moves to the only other cell in its current column and level, it will be in a row of three rooks, which contributes $1\da_{3,2}=1\cdot-1\cdot-3=3$ to the overall weight. Therefore, summing over all possible locations for the rightmost rook, in this specific example, yields $-1W+3W=2W$. Clearly we cannot guarantee partitions whose $m$-weight sums to zero if we restrict ourselves to moving a single rook within a fixed column in a level which contains two or more rooks. This illustrates the need to define $\mathcal{P}(F_0)$ as was done at the start of this section. Recall that, in general, $\mathcal{P}(F_0)$ is defined to be the set of file placements where every rook in level $l$ other than the leftmost one is allowed to move to any cell in its original column and level.

\begin{thm}
\label{partitions}
Let $F_0$ be a file placement that is not an $m$-level rook placement and $\mathcal{P}(F_0)$ the partition that contains $F_0$, then $$\sum_{F\in\mathcal{P}(F_0)}wt_m(F) = 0.$$
\end{thm}

\begin{proof}
Let level $l$ be the level containing a minimal number of rooks of those containing at least two rooks, as identified above. We proceed by induction on the number of rooks in level $l$. If $l$ contains exactly two rooks, we know $\sum_{F\in\mathcal{P}(F_0)}wt_m(F) = 0$ by Lemma~\ref{base case}.

Therefore, let us assume that level $l$ contains $n>2$ rooks. Since all the rooks outside of level $l$ remain unmoved across all file placements in $\mathcal{P}(F_0)$, we will denote their contribution to the $m$-weight as $W=wt_m(l^C)$ as was done in the proof of Lemma~\ref{base case}. Consider removing the rightmost rook in level $l$ to obtain a file placement with one fewer rook which we will denote as $\hat{F}$. The $m$-weight of this new placement is $wt_m(\hat{F}) = wt_m(l^C)\cdot 1\da_{n_1,m}\cdot1\da_{n_2,m}\cdot\ldots\cdot1\da_{n_m,m}$ where $n_i$ is the number of rooks in row $i$ of level $l$ once the rightmost rook in level $l$ has been removed.

Reintroducing the rook we removed back into the first row of level $l$, and the same column we removed it from, will change the $m$-weight to be $$wt_m(l^C)\cdot1\da_{n_1+1,m}\cdot1\da_{n_2,m}\cdot\ldots\cdot1\da_{n_m,m} = wt_m(\hat{F})\cdot(1-n_1m),$$ since $1\da_{n_1+1,m} = 1\da_{n_1,m}\cdot(1-n_1m)$. Similarly, if we put the rightmost rook in level $l$ back into the second row of level $l$, the $m$-weight of the resulting placement will be $$wt_m(l^C)\cdot1\da_{n_1,m}\cdot1\da_{n_2+1,m}\cdot\ldots\cdot1\da_{n_m,m} = wt_m(\hat{F})\cdot(1-n_2m).$$ Therefore, if we sum across all $m$ possible places to put the rook back, we get \begin{eqnarray*}
wt_m(\hat{F})\cdot\left ((1-n_1m)+(1-n_2m)+\ldots+(1-n_mm)\right )&=&wt_m(\hat{F})\cdot(m-(n_1+n_2+\cdots+n_m)m)\\
&=&wt_m(\hat{F})\cdot(m-(n-1)m)\\
&=&wt_m(\hat{F})\cdot(2m-nm)\\
&=&wt_m(\hat{F})\cdot m(2-n),
\end{eqnarray*}
because the sum of the number of rooks left in each row of the level after removing the rightmost rook will be the total number of rooks left in the level, which is $n-1$.

In addition to verifying our base case, that the $m$-weight will be $0$ if we allow $n=2$, this tells us that if we sum over all possible positions of the rightmost rook in level $l$, preserving which column it is in, then what we get will be a constant multiple of the $m$-weight of the placement where we simply remove the rightmost rook from level $l$. The constant will depend on the fixed values of $m$ and $n$, but it will be a constant over the partition $\mathcal{P}(F_0)$.

To complete the proof that $\sum_{F\in\mathcal{P}(F_0)}wt_m(F) = 0$, once again consider removing the rightmost rook in the specified level. If we let $\hat{F_0}$ denote the specific placement generated by removing the rightmost rook in the specified level of $F_0$, we get:
$$\sum_{F\in\mathcal{P}(F_0)}wt_m(F) = \sum_{\hat{F}\in\mathcal{P}(\hat{F_0})}\sum_{i=1}^m wt_m(\hat{F})\cdot (1-n_im)=m(2-n)\sum_{\hat{F}\in\mathcal{P}(\hat{F_0})}wt_m(\hat{F}) $$
By inductive hypothesis, $\sum_{\hat{F}\in\mathcal{P}(\hat{F_0})}wt_m(\hat{F}) = 0$. Therefore the entire sum $\sum_{F\in\mathcal{P}(F_0)}wt_m(F) = 0$ as desired.

\end{proof}

\section{Open Questions}
\label{sec:open}

While this provides a more illuminating explanation of why the sum of all $m$-weighted file placements that are not $m$-level rook placements on a singleton board must come out to zero, it leaves three major questions open that I can see. The first is the question of enumerating how many Ferrers boards have a given set of level numbers mentioned at the end of Section~\ref{sec:rp}. The second iswhether we can reduced the partition sizes. Finally, the third is whether a more combinatorial interpretation of this result is possible. Since we have already discussed the first question, let us move on to the second.

We know that the size of $\mathcal{P}(F)$ will be $m^{n-1}$ where $n$ is the least number of rooks in any level such that $F$ includes multiple rooks in that level. However, for large boards, $n$ could be quite large. So, the question is whether we can find a different partition of all the file placements such that the $m$-weight will be zero on each set in the partition, but the largest set in the partition is smaller than $m^{n-1}$?

When constructing $\mathcal{P}(F)$, we held the position of rooks outside of a single, specified, level constant. While this simplified the calculation of the $m$-weights, perhaps the partition sizes could be reduced by including placements where rooks change positions in multiple levels. However, there is always the consideration of an $m\times n$ board with a large $n$. Since that board only contains one level, any more efficient partitioning of the relevant file placements would almost necessarily be a refinement of the partitions presented in this paper. Alternatively, there is also the option of changing which column a given rook is in, but since the column doesn't affect the $m$-weight, this doesn't seem to increase the possibilities at first glance.

The final question is fairly straightforward. In enumerative combinatorics it is always nice when the numbers you are working with count something. Since weighted file placements can be negative, that presents an obstacle. However, since the objective is to show that a sum is zero, it is theoretically possible that the positively weighted file placements and the negatively weighted file placements are counting the same thing, and therefore must sum to zero together. A result of this nature would be interesting.


\end{document}